\documentclass[10pt,twoside,a4paper,openright,final]{article}
\usepackage[top=4.5cm, left=38mm, right=38mm,bottom=4.5cm]{geometry}
\usepackage{amssymb,amsmath,amsfonts,amscd,amsthm,verbatim}
\usepackage{tikz}
\newtheorem{lemma}{Lemma}[section]
\newtheorem{theorem}[lemma]{Theorem}

\newtheorem{proposition}[lemma]{Proposition}
\newtheorem{corollary}[lemma]{Corollary}

\newtheorem{remark}[lemma]{Remark}
\newtheorem{example}[lemma]{Example}

\DeclareMathOperator{\supp}{Supp}
\DeclareMathOperator{\can}{Can}
\DeclareMathOperator{\lcm}{lcm}
\DeclareMathOperator{\Log}{Log}

\newcommand{\HS}{\mathbb{H}_y}

\title{Boolean ideals and their varieties}

\author{Samuel Lundqvist}
\date{}                                          
\begin{document}
\maketitle


\begin{abstract}
We consider ideals in the ring $\mathbb{Z}_2[x_1,\ldots, x_n]$ that contain the polynomials $x_i^2 - x_i$ for $i = 1, \ldots, n$ and give various results related to the one-to-one correspondence between these ideals and the subsets of $\mathbb{Z}_2^n$. 
We also study the standard monomials with respect to the lexicographical ordering for these ideals and derive a distribution result.
\end{abstract}

\section{Introduction}

Suppose we are given polynomials $f_1,\ldots,f_m$ in $n$ variables over a field 
$\Bbbk$. The problem of describing the set of $\Bbbk$-rational solutions to the system $f_1= 0,\ldots, f_m = 0$ is one of the crucial problems in computational algebraic geometry. In this paper, we study this 
zero set when $\Bbbk = \mathbb{Z}_2,$ the field with two elements. 

The usual algebraic approach to the problem of finding the number of $\mathbb{Z}_2$-rational solutions is by means of Gr\"obner bases, and a number of papers in this direction has appeared over the last few years, for  instance, see \cite{bardet,polybori,newdev, gerdt, sato}.  
Here we present an alternative approach using fine Hilbert series of monomial ideals, and with no use of the theory of Gr\"obner bases.

We say that an ideal $I$ in $\mathbb{Z}_2[x_1,\ldots,x_n]$ is \emph{boolean} if $x_i^2 - x_i \in I$ for $i = 1, \ldots, n$. A \emph{boolean monomial} is a square free monomial in $\mathbb{Z}_2[x_1,\ldots,x_n]$ and a \emph{boolean polynomial} is a sum of different boolean monomials.
It is well known, see \cite[page 11]{atiyah}, that there is a one-to-one correspondence between boolean ideals and boolean polynomials. 
Indeed, each boolean ideal can be written uniquely as $(x_1^2 - x_1,\ldots, x_n^2 - x_n, f)$, where $f$ is boolean. We call such an $f$ the \emph{defining polynomial} of the ideal $(x_1^2 - x_1,\ldots, x_n^2 - x_n, f)$.
Along with the one-to-one correspondence between boolean ideals and subsets of $\mathbb{Z}_2^n$, we investigate properties of boolean ideals, especially their varieties, but also their standard monomials, in terms of their defining polynomial.

 In Section \ref{sec:algebra}, we derive simple formulas for quotients, sums, products and intersections of boolean ideals in terms of their defining polynomials. Also, we give a canonical form with respect to a boolean ideal without using the notion of Gr\"obner bases (or monomial orderings).
 
In Section \ref{sec:zeroset}, the main results of this paper is presented, and we give a detailed description of the correspondence between zero sets and the defining polynomials.
A surprising result is the following.
Consider a boolean polynomial $f$ in $\mathbb{Z}_2[x_1,\ldots,x_n]$ and regard each element in the rational zero set of $f$ as an exponent vector of a boolean monomial and add these boolean monomials together to construct a new boolean polynomial $\phi(f)$. We show in Theorem \ref{thm:fquad} that the fourth power of $\phi$ is the identity. For instance, the polynomial $x_1 x_2 + x_2 \in \mathbb{Z}_2[x_1,x_2]$ gives rise to the chain
$$x_1 x_2 +x_2 \mapsto x_1^1 x_2^1 + x_1^1 x_2^0 + x_1^0 x_2^0 = x_1 x_2 +x_1 + 1 \mapsto  x_1  \mapsto x_2 + 1 \mapsto x_1 x_2 + x_2$$
 when we apply $\phi$ repeatedly. 
 
 The purpose of this paper is not computational, but in Section \ref{sec:eff} we show that there are  
 instances of problems where the method proposed in this paper outperforms the Buchberger algorithm, motivating a further analysis of the computational aspects of the methods.

In Section \ref{sec:standard}, we give a recursive correspondence between a boolean polynomial and the lexicographical standard monomials of the associated boolean ideal, inspired by the Cerlienco-Mureddu correspondence \cite{cm}.
The correspondence is then used to give a recursive formula which, to a given boolean monomial, determines the number of ideals having this particular monomial as a standard monomial, and we show that this number is always odd.

\section{Notation}
 Points in $\mathbb{Z}_2^n$ will be denoted by the letters $p$ and $q$, and the $i$'th coordinate of the point $p$ is denoted by $p_i$. Sometimes $p_i$ and $q_j$ will be representing points, and in these cases, the coordinate of a point is written using double indices.  
  
 The symbols $\alpha$ and $\beta$ will be used to represent exponent vectors of boolean monomials, and as such, they will be elements in $\{0,1\}^n$. 
We will continue the abuse of notation from the introduction and regard points in $\mathbb{Z}_2^n$ as exponent vectors in $\{0,1\}^n$, and exponent vectors in  $\{0,1\}^n$ as points in $\mathbb{Z}_2^n$.

We evaluate a boolean polynomial $f$ at an affine point $p$ in $\mathbb{Z}_2^n$ by the notation $f(p)$ and we introduce the concept of a zero set for a boolean polynomial $f$, $V(f)$, defined as 
$$V(f) := V((x_1^2 - x_1, \ldots, x_n^2 - x_n, f)) = \{p \in \mathbb{Z}^n_2 \, : \, f(p) = 0\}.$$

We say that two polynomials $f,g \in \mathbb{Z}_2[x_1,\ldots, x_n]$ are equal as boolean polynomials if $f \equiv g$ modulo $(x_1^2 -x_1, \ldots, x_n^2 - x_n)$. This is expressed as "$f = g$ as boolean polynomials". Moreover, the expression "the boolean polynomial $f \cdot g$" is short notation for the boolean polynomial $h$ such that $f \cdot g \equiv h$ modulo $(x_1^2 -x_1, \ldots, x_n^2 - x_n)$.

If $A$ is a subset of $\mathbb{Z}_2^n$, we denote by $A^c$ the complement of 
$A$ in $\mathbb{Z}_2^n$, and given a boolean polynomial $f = m_1 +\cdots + m_s$, let $\supp(f) = \{m_1, \ldots, m_s\}$.

\section{Boolean polynomials, boolean ideals and subsets of $\mathbb{Z}_2^n$}

The one-to-one correspondences mentioned in the introduction are wide spread, but for completeness and future reference we state them in a lemma and give a short proof.  

\begin{lemma} \label{lemma:one-to-one-correspondences}
Let $\mathbb{Id}$ denote the set of boolean ideals in $n$ variables, let $\mathbb{P}$ denote the set of boolean polynomials in $n$ variables and let $\mathbb{V}$ denote the power set of $\mathbb{Z}_2^n$. Then the maps 
\begin{itemize}
\item $\mathbb{P} \to \mathbb{V},  f \mapsto V(f),$
 
\item $\mathbb{I} \to \mathbb{V},  I \mapsto V(I),$

\item $\mathbb{P} \to \mathbb{I},  f \mapsto (x_1^2 - x_1, \ldots, x_n^2 - x_n,f),$

\end{itemize}
are bijective and the defining polynomial of the boolean ideal $I = (x_1^2 - x_1, \ldots, x_n^2 - x_n, f_1, \ldots, f_s)$ is $(f_1+1)\cdots(f_s+1) + 1$.
\end{lemma}

\begin{proof}
We have already given a reference for the fact that the third map is bijective. The second map is bijective since boolean ideals are radical, indeed, $a^2 = a$ holds for all $a \in \mathbb{Z}_2[x_1,\ldots,x_n]/(x_1^2 - x_1, \ldots, x_n^2 - x_n)$.
We can then compose the third and the second map to show that the first map is one-to-one. 

Finally, let $I = (x_1^2 - x_1, \ldots, x_n^2 - x_n, f_1, \ldots, f_s)$. If $p \in V(I)$, then $((f_1+1) \cdots (f_s+1) + 1)(p) =  1 \cdots 1 + 1 = 0$ and if $p \notin V(I)$, then $f_i(p) = 1$ for some $i$, so 
$(f_1+1) \cdots (f_s+1) + 1)(p) = 1$. It follows that $V(I) = V((f_1+1) \cdots (f_s+1) + 1)$ and since the second map is bijective, we must have that $I = (x_1^2 - x_1, \ldots, x_n^2 - x_n, (f_1+1) \cdots (f_s+1) + 1).$ 
\end{proof}

A consequence of the bijective property of the first map is that the equation $V(f) = \mathbb{Z}_2^n$ has the unique solution $f = 0$ and that the equation $V(f) = \emptyset$ has the unique solution $f = 1$.

In this paper, the focus is on the first and the third map and their inverses, but before we begin, let us briefly discuss the second map and its relation to the SAT problem.

In Theoretical Computer Science, the SAT problem is the problem to determine whether a given boolean formula is satisfiable or not. 
Applications are found within a variety of areas, including complexity theory and formal verification. It is straightforward to interpret a boolean formula as a system of boolean polynomials such that there is a one-to-one correspondence between the the satisfiable assignments and the solutions to the system defined by the boolean polynomials. Thus, the special case 
$$ I \mapsto 
\begin{cases}
\text{false }& \text{if } V(I) = \{ \} \\
\text{true } & \text{if } V(I) \neq \{ \}
\end{cases}
$$
 of the second map can be used as a SAT solver, and the connection to the SAT problem is indeed a major motivation for the development of the Gr\"obner techniques for boolean ideals. 
 
Notice that since the defining polynomial of $I = (x_1^2 - x_1, \ldots, x_n^2 - x_n, f_1, \ldots, f_s)$ is $(f_1+1)\cdots(f_s+1) + 1$, it is possible to test if $V(I)$ is empty by checking if $(f_1+1)\cdots(f_s+1) + 1 = 1$ as boolean polynomials. However, from a computational point of view, a direct implementation of this approach fails to be effective due to a massive term expansion on the left hand side.

\subsection{A monomial order free canonical form, ideal membership and other ideal operations} \label{sec:algebra}
We will now study some fundamental algebraic properties of boolean ideals.

\begin{lemma} \label{lemma:intersectionvar}
Let $f$ and $g$ be boolean polynomials. Then
$$V(f)  \cap V(g) = V(f+g+fg).$$
\end{lemma}
\begin{proof}
Suppose that $p \in V(g)^c$. Then $(f+g+fg)(p)  = 1$. On the contrary, suppose that $p \in V(g)$. Then 
$(f+g+fg)(p)  = f(p)$. Hence, $p \in V(f+g+fg)$  if and only if $p \in V(f)  \cap V(g)$. 
\end{proof}

\begin{lemma} \label{lemma:rec}
Let $f$ and $g$ be boolean polynomials. Then 
$$V(f+g) = (V(f) \cap V(g)) \cup (V(f)^c \cap V(g)^c).$$
\end{lemma}
\begin{proof}
If $p \in V(f) \cap V(g)$, then $(f+g)(p) = f(p) + g(p) = 0$, so $p \in V(f+g)$.
If $p \in V(f)$ but $p \notin V(g)$ or vice versa, then $(f+g)(p) = 1$, so $p \notin V(f+g)$.
Finally, if $p \in V(f)^c \cap V(g)^c$, then $(f+g)(p) =  1 + 1 = 0$, so $p \in V(f+g)$.
\end{proof}

\begin{lemma} \label{lemma:complement}
Let $f$ be a boolean polynomial. Then
$$V(f+1) = V(f)^c.$$
\end{lemma}
\begin{proof} Immediate.
\end{proof}

Let $S$ be a ring and let $I$ be an ideal in $S$. A \emph{canonical form} with respect to $I$ is a function $\phi$ from $S$ to $S$ such that $\phi(a) = \phi(b)$ if and only if $a = b$ in $S/I$.

Suppose that $S$ is a polynomial ring and that $I$ is an ideal in $S$, and that we have computed a Gr\"obner basis $G$ of $I$ with respect to some monomial ordering $\prec$. Then the map from $S$ to $S$ defined as reduction modulo $G$ is a canonical form. 

However, different monomial orderings give rise to different canonical forms. 
But in our setting, it turns out to be possible to introduce a canonical form without any assumptions on monomial orderings.

\begin{proposition} \label{prop:canonicalform}

Let $I$ be a boolean ideal in $\mathbb{Z}_2[x_1,\ldots,x_n]$ and let $f$ be the defining boolean polynomial, i.e. $I = (x_1^2 - x_1, \ldots, x_n^2 - x_n, f)$.
Let $g$ be a polynomial in $\mathbb{Z}_2[x_1, \ldots, x_n]$. 
Then the map $\can_f:  \mathbb{Z}_2[x_1,\ldots,x_n] \to \mathbb{Z}_2[x_1,\ldots,x_n]$ defined by letting $\can_f(g)$ be the boolean polynomial $(f+1)\cdot g$ is
a canonical form with respect to $I$.
\end{proposition}

\begin{proof}

We need to show that $g_1+ I = g_2 + I$ if and only if $\can_f(g_1) = \can_f(g_2)$. 

By definition, $\can_f(g_1) = \can_f(g_2) $
is equivalent to 
$(g_1 + g_2)(f+1) = 0$ as boolean polynomials and $g_1+ I = g_2 + I$ is equivalent to $g_1 + g_2 \in I.$

Suppose that $g_1 + g_2 \in I$. Then $(g_1 + g_2)(p) = 0$ if $p \in V(f)$, and if $p \in V(f)^c$, then $(f+1)(p) = 0$ by Lemma \ref{lemma:complement}. Thus $(g_1 + g_2)(f+1)$ is identically zero on each point in 
$\mathbb{Z}_2^n$ and must be equal to zero by the one-to-one correspondence between boolean polynomials and subsets of $\mathbb{Z}_2^n$ given in Lemma \ref{lemma:one-to-one-correspondences}.

Suppose instead that $(g_1 + g_2)(f+1) = 0$ as boolean polynomials. Since $(f+1)(p) = 1$ if $p \in V(f)$, it follows that $(g_1+g_2)(p) = 0$ if $p \in V(f)$. Hence $g_1 + g_2 \in I$.
\end{proof}

\begin{proposition} \label{prop:equivalances}

Let $I$ be a boolean ideal with defining boolean polynomial $f$. Let $g$ be a polynomial in $\mathbb{Z}_2[x_1,\ldots,x_n]$ and let 
$J = (x_1^2 - x_1, \ldots, x_n^2 - x_n, g)$.
The following are equivalent.

\begin{enumerate}

\item $g \in I$.

\item $J \subseteq I.$

\item $V(I) \subseteq V(J)$.

\item $V(f) \subseteq V(g)$.

\item $f \cdot h = g$ as boolean polynomials for some polynomial $h$.

\item $f \cdot g = g$ as boolean polynomials.

\item $V(f) \cap V(g+1) = \{\}$.

\item $(f+1) \cdot g = 0$ as boolean polynomials.

\item $\can_f(g) = 0$ as boolean polynomials.

\end{enumerate}

\end{proposition}

\begin{proof}

\noindent

$(1) \Leftrightarrow (2)$: If $g \in I$, then $(x_1^2 - x_1, \ldots, x_n^2 - x_n, g) \subseteq I$. If $(x_1^2 - x_1, \ldots, x_n^2 - x_n, g) \subseteq I$, in particular, $g \in I.$

$(3) \Leftrightarrow (4)$: By definition.

$(4) \Rightarrow (2)$: It holds that $g(p) = 0$ for all $p \in V(f)$. Hence $g \in I.$ 

$(2) \Rightarrow (4)$: If $g \in I$ we have $g(p) = 0$ for all $p \in V(I)$. Hence $V(f) \subseteq V(g)$.

$(4) \Rightarrow (7)$: If $V(f) \subseteq V(g)$, then $V(f) \cap V(g)^c = \{ \}$. But 
$V(g)^c = V(g+1)$ by Lemma \ref{lemma:complement}. 

$(7) \Leftrightarrow (8):$ By Lemma \ref{lemma:intersectionvar}, 
$V(f) \cap V(g+1) = \{\}$ is equivalent to $V((f+1) \cdot g + 1) = \{\}$. By the one-to-one correspondence between boolean polynomials and subsets of $\mathbb{Z}_2^n$ in Lemma \ref{lemma:one-to-one-correspondences} this is equivalent to $(f+1) \cdot g +1 = 1 \Leftrightarrow (f+1) \cdot g = 0$ as boolean polynomials.

$(8) \Leftrightarrow (9):$ By definition.

$(6) \Leftrightarrow (8)$: Add $g$ on both sides in (6). 

$(6) \Rightarrow (5)$: The implication is trivial. 

$(5) \Rightarrow (4)$: Since $V(f \cdot h) = V(f) \cup V(h)$ it follows that $V(f) \subseteq V(g)$.

\end{proof}

The last result in this section is a formula which shows how to perform the common ideal operations on boolean ideals in terms of their defining polynomials.

\begin{theorem} \label{thm:idealop}
Let $I$ and $J$ be boolean ideals and let $f$ and $g$ be the defining 
polynomials of $I$ and $J$ respectively. Then 
\noindent

\begin{enumerate}

\item $I:J = (x_1^2 - x_1, \ldots, x_n^2 - x_n, 1 + g + fg).$ 

\item  $I + J =  (x_1^2 - x_1, \ldots, x_n^2 - x_n, f+g+fg).$

\item $IJ = I \cap J = (x_1^2 - x_1, \ldots, x_n^2 - x_n, fg).$
\end{enumerate}

\end{theorem}

\begin{proof}
From standard textbooks in computational algebra, see for instance \cite{cox}, we have the general equalities 
$I(V):I(W) = I(V \setminus W)$, $V(I+J)  = V(I) \cap V(J)$ and $V(IJ) = V(I \cap J) = V(I) \cup V(J)$, where $I(V)$ denotes the vanishing ideal with respect to the variety $V$. 
By the one-to-one correspondences in Lemma \ref{lemma:one-to-one-correspondences}, it is enough to show 
\begin{enumerate}
\item $V(f) \setminus V(g)=V(1+g+fg)$.
\item $V(f) \cap V(g) = V(f+g+fg)$.
\item $V(fg) = V(f) \cup V(g)$.
\end{enumerate}
The second equality follows from Lemma \ref{lemma:intersectionvar} and the third one is trivial.

Let us now prove the first equality. Take $p \in V(f)$. If $p \notin V(g)$, then $g(p) = 1$ and $p \in V(1+g+fg)$. If $p \in V(f)$ but $p \in V(g)$, then $(1+g+fg)(p) = 1$. Finally, if $p \notin V(f)$, then $(1+g+fg)(p) = 1.$ This shows that
$V(f) \setminus V(g)=V(1+g+fg)$.

\end{proof}

\begin{remark}

Notice that the defining polynomial of $I:J$ is equal to $\can_f(g)+1$.

\end{remark}

\begin{example}

Let $$I = (x_1^2 - x_1, x_2^2-x_2, x_3^2 - x_3, x_1x_2 + x_3, x_1 x_3 + x_2, x_3 + 1) \subseteq \mathbb{Z}_2[x_1,x_2,x_3].$$ Since $$(x_1x_2 + x_3 + 1) \cdot (x_1 x_3 + x_2 + 1) \cdot  (x_3 + 1 + 1) + 1 = x_1x_2x_3 + x_3 + 1$$
as boolean polynomials, we have 
$$ I = (x_1^2 - x_1, x_2^2 - x_2, x_3^2 - x_3, f),$$ where $f = x_1x_2x_3 + x_3 + 1$. 
Since $f \neq 1$, it follows that $V(f) \neq \{ \}$ by Lemma \ref{lemma:one-to-one-correspondences}. Let $g = 1+x_1 x_3$. Then $\can_f(g) = (x_1x_2x_3 + x_3)(1+x_1 x_3) = x_1x_3 + x_3$ as boolean polynomials, and since $ x_1x_3 + x_3 \neq 0$, it follows that 
$g \notin I$ by Proposition \ref{prop:equivalances}. 
However, the element $ h = 1+ x_3 \in I$, since
$\can_f(h) =  0$. Thus $V(x_1 x_2 x_3 + x_3 + 1) \subseteq V(x_3 + 1) .$

In fact, $V(f) = \{(1,0,1), (0,0,1), (0,1,1) \}$, $V(g) = \{(1,0,1), (1,1,1)\}$ and 
$V(h) = \{(0,0,1), (0,1,1), (1,0,1), (1,1,1)\}$, and since $V(f) \cup V(g) = V(h)$, we obtain the factorization $f g = h$.

Let $J = (x_1^2 - x_1, x_2^2-x_2, x_3^2 - x_3, h).$ We have
$I:J = (x_1^2 - x_1, x_2^2-x_2, x_3^2 - x_3, 1 + h + fh)$ by Theorem \ref{thm:idealop}. Since $1 + h + fh = 1$ as boolean polynomials, it follows that $I:J = (1)$ which is the expected result since $V(I) \subseteq V(J)$. 

\end{example}

\subsection{The relationship between a boolean polynomial and its zero set} \label{sec:zeroset}

We will now go deeper into the relationship between a boolean polynomial and its zero sets by using fine Hilbert series of monomial ideals.

Let $M$ be an $\mathbb{N}^n$-graded module over an $\mathbb{N}^n$-graded algebra, so that
$M = \oplus_{\alpha \in N^n} M_{\alpha}$. The \emph{fine 
Hilbert series} of $M$ is
$$\mathbb{H}_t(M) = \sum_{\alpha \in \mathbb{N}^n} \dim_{\Bbbk} M_{\alpha}t^{\alpha}.$$

It will be convenient to work with two different algebras, thus we introduce the algebra 
$\mathbb{Q}[y_1,\ldots,y_n]/(y_1^2, \ldots, y_n^2)$. This algebra is $\mathbb{N}^n$-graded, 
and we shall use it in the context of fine Hilbert series.

There is a one-to-one correspondence between monomials in $$\mathbb{Q}[y_1,\ldots,y_n]/(y_1^2, \ldots, y_n^2)$$ and boolean monomials in $$\mathbb{Z}_2[x_1,\ldots,x_n].$$ Indeed, every monomial $x^{\alpha}$ corresponds to a monomial $y^{\alpha}$ and vice versa.

The $x$-logarithm of a monomial $x^p$ is $\Log_x(x^p) = p$ and the $x$-logarithm of a boolean polynomial $f=x^{\alpha_1}+ \cdots + x^{\alpha_s}$ is $\Log_x(f) = \{\Log(x^{\alpha_1}),\ldots, \Log(x^{\alpha_s})\}= \{\alpha_1, \ldots, \alpha_s\}.$ It will be convenient to work also with a $y$-logarithm. The $y$-logarithm is defined in the same way for polynomials in $\mathbb{Q}[y_1,\ldots,y_n]/(y_1^2, \ldots, y_n^2)$ with coefficients in $\{0,1\}$, i.e. polynomials $y^{\alpha_1}+ \cdots + y^{\alpha_s}$ such that $x^{\alpha_1} + \cdots + x^{\alpha_s}$ is a boolean polynomial. Since it will be clear from the context which logarithm we mean, we will simply write $\Log$ instead of $\Log_x$ or $\Log_y$. 

We
define $x^{\{p_1, \ldots, p_s \}}$ as $x^{p_1} + \cdots + x^{p_s}$, where each $p_i$ is a point in $\mathbb{Z}_2^n$ (but as an exponent regarded as a point in $\{0,1\}^n$).

The exponent and the logarithm will together be used to give the map from polynomials in $\mathbb{Q}[y_1,\ldots,y_n]/(y_1^2, \ldots, y_n^2)$ with coefficients in $\{0,1\}$ 
to the set of boolean polynomials. Indeed, if $f$ is a polynomial in $\mathbb{Q}[y_1,\ldots,y_n]/(y_1^2, \ldots, y_n^2)$  with coefficients in $\{0,1\}$, 
then $x^{\Log(f)}$ is the corresponding boolean polynomial. The map will be needed in the context of evaluation, the reason being that if $f \in \mathbb{Q}[y_1,\ldots,y_n]/(y_1^2, \ldots, y_n^2)$ with coefficients in $\{0,1\}$, and $p \in\mathbb{Z}_2^n$,  then $f(p)$ is not defined, but $x^{\Log(f)}(p)$ is.

\begin{lemma} \label{lemma:divide}
$x^q (p) = \begin{cases}
1 &\text{if } x^q | x^p \\
0 &\text{otherwise}
\end{cases}$
\end{lemma}
\begin{proof}
We have $x^q (p) = 1$ if and only if $q_i = 1$ implies $p_i = 1$, which is equivalent to the statement in the lemma.
\end{proof}

\begin{proposition} \label{prop:even} Let
$f = x^{\alpha_1} + x^{\alpha_s}$ be a boolean polynomial in $\mathbb{Z}_2[x_1,\ldots,x_n]$. Then 
$$V(f) = \{p \, :\, \text{ the number of } x^{\alpha_i} \text{ that is divisible by $x^p$ is even}\}.$$
\end{proposition}
\begin{proof}
Let $t$ denote the number of $x^{\alpha_i}$ that is divisible by $x^p$. Then $f(p) = \underbrace{1 + \cdots + 1}_{t \text{ times}}$ by Lemma \ref{lemma:divide}. Hence $f(p) = 0$ if and only if $t$ is even.
\end{proof}

\begin{theorem} \label{thm:explicit1}

Let $f = x^{\alpha_1} + \cdots + x^{\alpha_s}$ be a boolean polynomial in $\mathbb{Z}_2[x_1,\ldots,x_n]$. Then
$$V(x^{\alpha_1} + \cdots + x^{\alpha_s}) = \Log(h_0 - h_1 + h_2 - h_3 +  \cdots + (-1)^{s} h_s)$$
where
\begin{align*}
h_0& = \HS((1)),\\
\nonumber
h_1& = \HS((y^{\alpha_1},\ldots, y^{\alpha_s})),\\
\nonumber
h_2& =  \HS((\lcm(y^{\alpha_1},y^{\alpha_2}), \lcm(y^{\alpha_1},y^{\alpha_3}),\ldots,\lcm(y^{\alpha_{s-1}}, y^{\alpha_s})), \\
\nonumber
h_3& =  \HS((\lcm(y^{\alpha_1},y^{\alpha_2},y^{\alpha_3}), \lcm(y^{\alpha_1},y^{\alpha_3},y^{\alpha_4}),\ldots,\lcm(y^{\alpha_{s-2}},y^{\alpha_{s-1}}, y^{\alpha_s})), \\
\nonumber
\vdots \\
\nonumber
h_s& =  \HS((\lcm(y^{\alpha_1},\ldots, y^{\alpha_s})),
\end{align*}

and where all the ideals are in the $\mathbb{N}^n$-graded algebra $\mathbb{Q}[y_1,\ldots,y_n]/(y_1^2, \ldots,y_n^2)$.

\end{theorem}

\begin{proof}
By Proposition \ref{prop:even}, it is enough to show that the monomial $y^{\beta}$ occurs with coefficient $1$ in the expression $h_0 - h_1 + h_2 - h_3 +  \cdots + (-1)^{s} h_s$ if the number of $x^{\alpha_i}$ that is divisible by $x^{\beta}$ is even and with coefficient $0$ otherwise.

Suppose that  the number of $x^{\alpha_i}$ that is divisible by $x^{\beta}$ is equal to $r$. Then $y^{\beta}$ is not contained as a summand in the expressions 
$h_{r+1}, \ldots, h_s$, but it occurs exactly once in each $h_i$, for $i = 0, \ldots, r$. It follows that the coefficient in front of $y^{\beta}$ in $h_0 - h_1 + h_2 - h_3 +  \cdots + (-1)^{s} h_s$ is zero if $r$ is odd and one otherwise.
\end{proof}

\begin{example}

Let $f = 1 + x_1x_5+x_2x_4+x_2x_4x_5$ be a boolean polynomial in the five variables $x_1, x_2, x_3, x_4, x_5$. Then 
$$h_1 = \HS((1,x_1x_5,x_2x_4,x_2x_4x_5)) =\HS((1)) =\sum_{p \in \mathbb{Z}_2^n} y^p.$$ 
Since 
\begin{align*}
(\lcm(1,x_1x_5), \lcm(1,x_2 x_4), \lcm(1,x_2 x_4 x_5),
 \lcm(x_1x_5, x_2x_4),\\
\lcm(x_1x_5,x_2 x_4 x_5),
 \lcm(x_2 x_4, x_2 x_4 x_5))
\nonumber
 = (x_1 x_5, x_2 x_4),
\end{align*}
it follows that
\begin{align*}
h_2 & = \HS((y_1 y_5, y_2 y_4))\\
 &=y_1 y_5 + y_1 y_2 y_5+ y_1 y_3 y_5+y_1 y_4 y_5+ y_1 y_2 y_3 y_5+ y_1 y_2 y_4 y_5+ y_1 y_3 y_4 y_5+\\
\nonumber
&y_1 y_2 y_3 y_4 y_5+ y_2 y_4 + y_1 y_2 y_4 + y_2 y_3 y_4 + y_2 y_4 y_5 + y_1 y_2 y_3 y_4 + y_2 y_3 y_4 y_5 .\\
\end{align*}
Similarly, $$h_3 = y_2 y_4 y_5 + y_1 y_2 y_4 y_5 + y_2 y_3 y_4 y_5 + y_1 y_2 y_3 y_4 y_5$$
and 
$$h_4 = y_1 y_2 y_4 y_5 + y_1 y_2 y_3 y_4 y_5,$$
which gives us 
\begin{align*}
&h_0 - h_1 + h_2 - h_3 + h_4 \\
= & y_1 y_5 + y_1 y_2 y_5+ y_1 y_3 y_5+y_1 y_4 y_5+ y_1 y_2 y_3 y_5+ y_1 y_2 y_4 y_5+ \\
\nonumber
&y_1 y_3 y_4 y_5+ y_1 y_2 y_3 y_4 y_5+y_2 y_4 + y_1 y_2 y_4 + y_2 y_3 y_4  + y_1 y_2 y_3 y_4 .\\
\end{align*}
Taking the logarithm of this polynomial, we obtain 
\begin{eqnarray*}
V(f) = \{(1,0,0,0,1), (1,1,0,0,1), (1,0,1,0,1), (1,0,0,1,1), (1,1,1,0,1), (1,1,0,1,1), \\
(1,0,1,1,1), (1,1,1,1,1), (0,1,0,1,0), (1,1,0,1,0), (0,1,1,1,0), (1,1,1,1,0)\},
\end{eqnarray*}
which concludes the example.

\end{example}

From the principle of inclusion and exclusion (see for instance \cite{froberg}), it follows that  

\begin{align} \label{eq:inklexkl}
\HS((y^{\alpha_1},\ldots, y^{\alpha_s})) = 
& \HS((y^{\alpha_1})) + \cdots + \HS((y^{\alpha_s}))
\nonumber
-(\HS((\lcm(y^{\alpha_1},y^{\alpha_2})))  \\ 
\nonumber
&+ \HS((\lcm(y^{\alpha_1},y^{\alpha_3})))
+ \cdots +  
\HS((\lcm(y^{\alpha_{s-1}}, y^{\alpha_s}))))\\
&+ \cdots + (-1)^s \HS(\lcm(y^{\alpha_1},\ldots, y^{\alpha_s})).
\end{align}
Hence the variety can be expressed as a linear combination of fine Hilbert series of principal ideals. 
Although it is possible to determine the coefficients combinatorially by combining Theorem \ref{thm:explicit1} and (\ref{eq:inklexkl}), we choose to use induction on $s$. For the induction step we need two lemmas.

\begin{lemma}\label{lemma:intersectformula}
Let $\{I_k\}$ and $\{J_h\}$ be finite families of ideals in 
$\mathbb{Q}[y_1,\ldots,y_n]/(y_1^2,\ldots, y_n^2)$. Let $f$ and $g$ be boolean 
polynomials in $\mathbb{Z}_2[x_1,\ldots,x_n]$. Suppose that
$$V(f) = \Log(\sum_k c_k\HS(I_k))$$ and that
$$V(g) =   \Log(\sum_h d_h\HS(J_h)).$$
Then 
$$V(f) \cap V(g) = \Log (\sum_{k,h} c_k d_h \HS(I_k \cap J_h)).$$ 
\end{lemma}
\begin{proof}
Let $p$ be any point in $\mathbb{Z}_2^n$. Let $i_1, \ldots, i_s$ be the set of indices such that 
$y^p \in I_{i_1},\ldots, y^p \in I_{i_s}$ and let similarly $j_1, \ldots, j_t$ be the set of indices such that 
$y^p \in J_{j_1},\ldots, y^p \in J_{j_t}$. 

It follows that
$y^p \in I_k \cap J_h$ if and only if  $I_k$ is one of $I_{i_1},\ldots, I_{i_s}$ and $J_h$ is one of $J_{j_1}, \ldots, J_{j_t}$. Hence the coefficient in front of 
 $y^p$ in $\sum_{k,h} c_k d_h \HS(I_k \cap J_h)$ is 
 $$c_{i_1} d_{j_1} + c_{i_1} d_{j_2} + \cdots + c_{i_1} d_{j_t} + \cdots + c_{i_s} d_{j_t} = (c_{i_1} + \cdots + c_{i_s}) \cdot (d_{j_1} + \cdots + d_{j_t}).$$
 
Now, if $p \in V(f)$, then $c_{i_1} + \cdots + c_{i_s} = 1$ and if  $p \notin V(f)$, then $c_{i_1} + \cdots + c_{i_s} = 0$. In the same way it holds that if 
$p \in V(g)$, then $d_{i_1} + \cdots + d_{i_t} = 1$ and if  $p \notin V(g)$, then $d_{j_1} + \cdots + d_{j_t} = 0$. Hence 

$$(c_{i_1} + \cdots + c_{i_s}) \cdot (d_{j_1} + \cdots + d_{j_t}) =
\begin{cases}
1 & \text{if } p \in V(f) \cap V(g) \\
0 & \text{otherwise.} \\
\end{cases}$$

Thus the logarithm of $(\sum_{k,h} c_k d_h \HS(I_k \cap J_h))$ is well defined and $V(f) \cap V(g) = \Log (\sum_{k,h} c_k d_h \HS(I_k \cap J_h))$.

\end{proof}

\begin{lemma} \label{lemma:hilbertcomplement}
Let $f$ be a boolean polynomial in $\mathbb{Z}_2[x_1,\ldots,x_n].$ 
Suppose that $V(f) = \Log(\sum_k c_k\HS(I_k))$. Then $$V(f)^c =  \Log(\HS((1)) - \sum_k c_k\HS(I_k)).$$

\end{lemma}

\begin{proof}
$\Log(\HS((1)) - \sum_k c_k\HS(I_k)) = \mathbb{Z}_2^n \setminus V(f) = V(f)^c.$
\end{proof}

\begin{theorem} \label{thm:explicit2}

Let $f = x^{\alpha_1} + \cdots + x^{\alpha_s}$ be a boolean polynomial in $\mathbb{Z}_2[x_1,\ldots,x_n]$. Then
$$V(f) = \Log(g_0 + \frac{1}{2}\sum_{1\leq i \leq s} (-2)^{i} g_i).$$
where $$g_0 = \HS((1)),$$ and for $i>0$, $$g_i = \sum_{1 \leq j_1 < \cdots < j_i \leq s} \HS(\lcm(y^{\alpha_{j_1}},\ldots,y^{\alpha_{j_i}}).$$

\end{theorem}

\begin{proof}

Induction over $s$. The base case follows since 
$$V(x^{\alpha_1}) = 
\Log (\HS((1)) - \HS((y^{\alpha_1})))$$
by Theorem \ref{thm:explicit1}.

Assume that the theorem holds true for $s = t$. 
By the induction assumption, it holds that 
\begin{align*}
&V(x^{\alpha_1} + \cdots + x^{\alpha_t}) = 
 \Log(\HS((1))  -2^0 \sum_{1 \leq i \leq s} \HS(y^{\alpha_i}) \\  
 &+ 2^1 \sum_{1\leq i < j \leq s} \HS(\lcm(y^{\alpha_i},y^{\alpha_j})) +  \cdots + (-1)^s2^{s-1} \HS(\lcm(y^{\alpha_1},\ldots, y^{\alpha_t}))).
\end{align*}
Another use of Theorem \ref{thm:explicit1} gives us $$V(x^{\alpha_{t+1}}) = \Log (\HS((1)) - \HS(y^{\alpha_{t+1}}))$$
so we can apply Lemma \ref{lemma:intersectformula} to obtain  
\begin{align*}
&V(x^{\alpha_1} + \cdots + x^{\alpha_t}) \cap V(x^{\alpha_{t+1}}) = \\
&\Log(\HS((1)) -2^0 \sum_{1 \leq i \leq s} \HS(y^{\alpha_i})  + 
2^1 \sum_{1\leq i < j \leq s} \HS(\lcm(y^{\alpha_i},y^{\alpha_j})) \\
& +  \cdots + (-1)^s2^{s-1} \HS(\lcm(y^{\alpha_1},\ldots, y^{\alpha_t}))\\
&-(\HS((y^{\alpha_{t+1}})) - 2^0 \sum_{1 \leq i \leq s} \HS(y^{\alpha_i},y^{\alpha_{t+1}}) + 2^1  \sum_{1\leq i < j \leq s} \HS(\lcm(y^{\alpha_i},y^{\alpha_j},y^{\alpha_{t+1}})) \\
&+ \cdots + (-1)^s2^{s-1} \HS(\lcm(y^{\alpha_1},\ldots, y^{\alpha_t}, y^{\alpha_{t+1}})).\\
\end{align*}
By Lemma \ref{lemma:hilbertcomplement} and the induction assumption, we have
\begin{align*}
V(x^{\alpha_1} + \cdots + x^{\alpha_t})^c&
= \Log(2^0 \sum_{1 \leq i \leq s} \HS(y^{\alpha_i})  -
2^1 \sum_{1\leq i < j \leq s} \HS(\lcm(y^{\alpha_i},y^{\alpha_j}))\\
&  +  \cdots + (-1)^{s-1} 2^{s-1} \HS(\lcm(y^{\alpha_1},\ldots, y^{\alpha_t})).
\end{align*}
By Lemma \ref{lemma:hilbertcomplement} and Theorem \ref{thm:explicit1}, it follows that
$$
V(x^{\alpha_{t+1}})^c = \Log (\HS(y^{\alpha_{t+1}})),
$$
so we can use Lemma  \ref{lemma:intersectformula} once more to obtain 
\begin{align*}
&V(x^{\alpha_1} + \cdots + x^{\alpha_t})^c \cap V(x^{\alpha_{t+1}})^c 
= 2^0 \sum_{1 \leq i \leq s} \HS(y^{\alpha_i},y^{\alpha_{t+1}}) \\
&- 2^1  \sum_{1\leq i < j \leq s} \HS(\lcm(y^{\alpha_i},y^{\alpha_j},y^{\alpha_{t+1}}))
+ \cdots + (-1)^{s-1} 2^{s-1} \HS(\lcm(y^{\alpha_1},\ldots, y^{\alpha_t}, y^{\alpha_{t+1}})).\\
\end{align*}
We are finally ready to use the identity 
\begin{align*}
V(x^{\alpha_1}+\cdots + x^{\alpha_{t+1}}) = V(x^{\alpha_1} + \cdots + x^{\alpha_t}) \cap V(x^{\alpha_{t+1}}) \\
  \cup V(x^{\alpha_1} + \cdots + x^{\alpha_t})^c \cap V(x^{\alpha_{t+1}})^c
\end{align*}
derived in Lemma  \ref{lemma:rec}. 
Indeed, 
\begin{align*}
&V(x^{\alpha_1} + \cdots + x^{\alpha_t}) \cap V(x^{\alpha_{t+1}}) \cup V(x^{\alpha_1} + \cdots + x^{\alpha_t})^c \cap V(x^{\alpha_{t+1}})^c\\
&= 
\Log(\HS((1)) -2^0 \sum_{1 \leq i \leq s} \HS((y^{\alpha_i}))  + 
2^1 \sum_{1\leq i < j \leq s} \HS((\lcm(y^{\alpha_i},y^{\alpha_j}))) \\
& +  \cdots + (-1)^s2^{s-1} \HS((\lcm(y^{\alpha_1},\ldots, y^{\alpha_t})))\\
&-(\HS((y^{\alpha_{t+1}})) - 2^0 \sum_{1 \leq i \leq s} \HS((y^{\alpha_i},y^{\alpha_{t+1}})) + 2^1  \sum_{1\leq i < j \leq s} \HS((\lcm(y^{\alpha_i},y^{\alpha_j},y^{\alpha_{t+1}}))) \\
&+ \cdots + (-1)^s2^{s-1} \HS((\lcm(y^{\alpha_1},\ldots, y^{\alpha_t}, y^{\alpha_{t+1}}))))\\
&+
2^0 \sum_{1 \leq i \leq s} \HS(y^{\alpha_i},y^{\alpha_{t+1}}) - 2^1  \sum_{1\leq i < j \leq s} \HS((\lcm(y^{\alpha_i},y^{\alpha_j},y^{\alpha_{t+1}}))) \\
&+ \cdots + (-1)^{s-1} 2^{s-1} \HS((\lcm(y^{\alpha_1},\ldots, y^{\alpha_t}, y^{\alpha_{t+1}}))))\\
&= \Log(\HS((1)) - 2^0 \sum_{1 \leq i \leq s+1} \HS((y^{\alpha_i}))  + 
2^1 \sum_{1\leq i < j \leq s+1} \HS((\lcm(y^{\alpha_i},y^{\alpha_j}))) \\
& +  \cdots + (-1)^{s+1}2^{s} \HS((\lcm(y^{\alpha_1},\ldots, y^{\alpha_{t+1}})))).
\end{align*}
which concludes the proof.

\end{proof}

\begin{corollary} \label{cor:main}
Let $I$ be a boolean ideal in $\mathbb{Z}_2[x_1,\ldots,x_n]$ with defining polynomial $x^{\alpha_1} + \cdots + x^{\alpha_s}$. Then
$$|V(I)| =  2^n + d_1 + d_2 +  \cdots + d_s,$$
where $$d_i = \frac{1}{2}(-2)^{i}  \sum_{1 \leq j_1 < \cdots < j_i \leq s} 2^{n-\deg(\lcm(x^{\alpha{_{j_1}}},\ldots,x^{\alpha_{j_i}}))}.$$
\end{corollary}
\begin{proof}
The corollary follows from Theorem \ref{thm:explicit2} since the number of monomials in $\HS(\lcm(y^{\alpha_{j_1}},\ldots,y^{\alpha_{j_i}}))$ is $2^{n-\deg(\lcm(x^{\alpha{_{j_1}}},\ldots,x^{\alpha_{j_i}}))}.$
\end{proof}

The next corollary is a direct consequence of Theorem \ref{thm:explicit2}, Corollary \ref{cor:main} and the one-to-one correspondences in Lemma \ref{lemma:one-to-one-correspondences}.

\begin{corollary}\label{cor:sys2}

Let $I = (x_1^2 - x_1, \ldots, x_n^2 - x_n,f_1,\ldots, f_s)$ be an ideal in 
$\mathbb{Z}_2[x_1,\ldots,x_n]$. Let $x^{\alpha_1} + \cdots + x^{\alpha_s}$ be the boolean polynomial such that $x^{\alpha_1} + \cdots + x^{\alpha_s} = (f_1 + 1) \cdots (f_s + 1) + 1$  as boolean polynomials.
Then 
$$V(I) = \Log(g_0 + \frac{1}{2}\sum_{1\leq i \leq s} (-2)^{i} g_i),$$
and
$$|V(I)| =  2^n + d_1 + d_2 +  \cdots + d_s,$$
where $$g_0 = \HS((1)),$$ 
$$g_i = \sum_{1 \leq j_1 < \cdots < j_i \leq s} \HS(\lcm(y^{\alpha_{j_1}},\ldots,y^{\alpha_{j_i}}) \text{ when $i>0$.}$$
and  $$d_i = \frac{1}{2}(-2)^{i}  \sum_{1 \leq j_1 < \cdots < j_i \leq s} 2^{n-\deg(\lcm(x^{\alpha{_{j_1}}},\ldots,x^{\alpha_{j_i}}))}.
$$

\end{corollary}

The next corollary to Theorem \ref{thm:explicit2} requires a short proof.

\begin{corollary} \label{cor:homo}
Let $f = x^{\alpha_1} + \cdots + x^{\alpha_s}$ be a boolean polynomial in $\mathbb{Z}_2[x_1,\ldots,x_n]$ and let $\pi$ be the natural homomorphism from 
$\mathbb{Z}[y_1,\ldots,y_n]/(y_1^2, \ldots, y_n^2)$ to $\mathbb{Z}_2[y_1,\ldots,y_n]/(y_1^2, \ldots, y_n^2)$. 
Then 
$$V(f) = \Log(\pi(\HS((1)) - (\HS(y^{\alpha_1}) + \cdots + \HS(y^{\alpha_s})))).$$

\end{corollary}

\begin{proof}
By Theorem \ref{thm:explicit2}, $$V(f) = \Log(g_0 + \frac{1}{2}\sum_{1\leq i \leq s} (-2)^{i} g_i),$$ with the $g_i$'s being defined in the referred theorem.
When $i > 1$, the coefficient in front of each $g_i$ is a multiple of $2$, and thus
a monomial has coefficient $1$ in 
$$g_0 + \frac{1}{2}\sum_{1\leq i \leq s} (-2)^{i} g_i$$
if and only if it has coefficient $1$ in  
$$\pi(g_0 - g_1).$$ 
Thus $$\Log(g_0 + \frac{1}{2}\sum_{1\leq i \leq s} (-2)^{i} g_i) = 
 \Log(\pi((\HS((1)) - (\HS(y^{\alpha_1}) + \cdots + \HS(y^{\alpha_s})))).$$
\end{proof}

Before we continue we need to introduce more notation. 
Let $X$ be a finite collection of points in some space. A function $S_p$ from $X$ to $\{0,1\}$ such that $S_p(p) = 1$ and $S_p(q) = 0$ when $p \neq q$ is called a separator of $p$ with respect to $X$.

\begin{lemma} \label{lemma:sep}
Let $p$ be a point in $\mathbb{Z}_2^n$. Then
$$S_p = x^{\Log \HS  (y^p)}$$ is a separator of $p$ with respect to $\mathbb{Z}_2^n$.
\end{lemma}

\begin{proof}
The only term in $S_p$ that divides $x^p$ is $x^p$ itself. Hence $x^p(p) = 1$, while $x^q(p) = 0$ for all the remaining monomials $x^q$ in $S_p$ by Lemma \ref{lemma:divide}, so 
$S_p(p) = 1.$

Let $q \neq p$. If $x^q \notin (x^p)$, then no elements in $(x^p)$ divides $x^q$, so 
$S_p(q) = 0$ in this case. Suppose instead that $x^q \in (x^p)$. 
Then $x^q = x^p \cdot m$ for some $m \neq 1$. It follows that 
$x^p \cdot m' | x^q$, for all submonomials $m'$ of $m$. There are $2^{\deg(m)}$ submonomials of $m$ and hence there are $2^{\deg(m)}$ monomials in $(x^p)$ which evaluate to 
one in the point $q$. The remaining monomials evaluate to zero. Since $m \neq 1$, $2^{\deg(m)}$ is an even integer, so $S_p(q) = 0$ and the proof of the lemma is complete.

\end{proof}

\begin{remark} \label{remark:sep}
Notice that $V(S_p +1) = \{p\}$.
\end{remark}

\begin{lemma} \label{lemma:vanishing}
Let $f=x^{\alpha_1}+ \cdots + x^{\alpha_s}$ be a boolean polynomial in $\mathbb{Z}_2[x_1,\ldots,x_n]$. Then

$$V(x^{V(f)}) = \begin{cases}
  \Log(f)^c \cup \{(0,\ldots,0)\} & \text{if } 1 \in \supp(f)  \\
   \Log(f)^c \setminus \{(0,\ldots,0)\} & \text{if } 1 \notin \supp(f).
\end{cases}$$
\end{lemma}

\begin{proof}

By Corollary \ref{cor:homo},
$V(f) = \Log(\pi (g_0 - g_1)) = \Log(\pi(g_0) + \pi(g_1)).$
Thus $x^{V(f)} = x^{\Log(\pi(g_0) + \pi(g_1))}.$

Suppose that $1 \in \supp(f)$. Assume, with no loss of generality, that  $x^{\alpha_1} = 1$. Then 
$\HS((y^{\alpha_1})) = \HS((1))$ and $\pi(g_0) + \pi(g_1) = \pi(\HS(y^{\alpha_2})) + \cdots + \pi(\HS(y^{\alpha_s})).$ 
By Lemma \ref{lemma:sep}, $x^{\Log(\HS(y^{\alpha_i}))}$ is a separator of $\alpha_i$.  Thus 
$$x^{\Log( \pi(g_0) + \pi(g_1))}(p) = 
\begin{cases}
1 & \text{ if } p \in \{\alpha_2, \ldots, \alpha_s\}\\
0 & \text {otherwise}.
\end{cases}
$$

Suppose instead that $1 \notin \supp(f)$. Then 
$\pi(g_0) + \pi(g_1) = \pi(\HS((1))) +\pi(\HS((y^{\alpha_1})))  + \cdots + \pi(\HS((y^{\alpha_s}))).$ With a similar argument as in the first part, we get that 
$$x^{\Log(\pi(g_0) + \pi(g_1))}(p) = 
\begin{cases}
1 & \text{ if } p \in \{\alpha_1, \alpha_2, \ldots, \alpha_s, (0,\ldots,0)\}\\
0 & \text {otherwise}.
\end{cases}
$$

\end{proof}

\begin{lemma} \label{lemma:phi}
Let $f$ be a boolean polynomial. Then 
$$V(f + \sum_{p \in \mathbb{Z}_2^n} x^p + 1) = 
\begin{cases}
V(f)^c \setminus \{(0,\ldots,0)\} & \text{ if } 1 \in \supp(f)\\
V(f)^c \cup \{(0,\ldots,0)\} & \text{ if } 1 \notin \supp(f)\\
\end{cases}
$$
\end{lemma}
\begin{proof}
We have $V(f + 1) = V(f)^c$ by Lemma \ref{lemma:complement}. By Lemma \ref{lemma:sep}, $\sum_{p \in \mathbb{Z}_2^n} x^p $ is a separator of $(0,\ldots,0)$, so by Remark \ref{remark:sep}, $V(\sum_{p \in \mathbb{Z}_2^n} x^p + 1) = \{(0,\ldots,0)\}$. It follows from Lemma  \ref{lemma:rec} that $V(f  + \sum_{p \in \mathbb{Z}_2^n} x^p + 1) = 
(V(f) \cap \{(0,\ldots,0)\}) \cup (V(f)^c \cap (\mathbb{Z}_2^n \setminus \{(0,\ldots,0)\})).$

If $1 \in \supp(f)$, then $(0,\ldots,0) \in V(f)^c$ and $(0,\ldots,0) \notin V(f)$, so 
$V(f + \sum_{p \in \mathbb{Z}_2^n} x^p + 1) = V(f)^c \setminus \{(0,\ldots,0)\}$.

If $1 \notin \supp(f)$, then $(0,\ldots,0) \notin V(f)^c$ and $(0,\ldots,0) \in V(f)$, so 
$V(f + \sum_{p \in \mathbb{Z}_2^n} x^p + 1) = \{(0,\ldots,0 \} \cup V(f)^c. $
\end{proof}

We have used the variety $V$ as a map from the set of boolean ideals to the power set of $\mathbb{Z}_n^2$, but also from the set of boolean polynomials to the power set of $\mathbb{Z}_n^2$.
For the next theorem we need to introduce a map from $\mathbb{Z}_n^2$ to the set of boolean polynomials, by sending $V(f)$ to $f$. We denote this bijective map by $V^{-1}.$

\begin{theorem} \label{thm:invalg}

Let $f$ be a boolean polynomial. Then
$$V(f) = \Log( V^{-1}(\Log(f)) + \sum_{p \in \mathbb{Z}_2^n} x^p + 1).$$
\end{theorem}
\begin{proof}
By taking the exponent followed by the variety $V$ on both sides we get the equivalent statement

$$V(x^{V(f)}) =  V(V^{-1}(\Log(f)) + \sum_{p \in \mathbb{Z}_2^n} x^p + 1).$$
We will now rewrite the right hand side and show that it agrees with the left hand side.
Let $g = V^{-1}(\Log(f)).$ Then $1 \in \supp(f)$ if and only if $1 \notin \supp(g)$. Applying Lemma \ref{lemma:phi} on
the right hand side we get
$$V(g+\sum_{p \in \mathbb{Z}_2^n} x^p + 1) =  \begin{cases}
V(g)^c \setminus \{(0,\ldots,0)\} & \text{ if } 1 \notin \supp(f)\\
V(g)^c \cup \{(0,\ldots,0)\} & \text{ if } 1 \in \supp(f).
\end{cases}
$$
Since $V(g)^c = \Log(f)^c$, this expression is equal to $V(x^{V(f)})$ by Lemma \ref{lemma:vanishing}.

\end{proof}

Theorem \ref{thm:invalg} gives an "inverse" way to compute the variety of a boolean polynomial. Indeed, to determine the variety of a boolean polynomial $f$, we can first take the logarithm of this polynomial, which is the zero set of some polynomial $g$. The logarithm of the sum $g + \sum_{p \in \mathbb{Z}_2^n} x^p + 1$ is then equal to the variety of $f$. 

\begin{example}
Let $$f = 1 + x_2 + x_3 + x_1 x_2 + x_1 x_3 + x_2 x_3 + x_1 x_2 x_3$$ be a boolean polynomial in three variables $x_1,x_2$ and $x_3.$
We have 
$V^{-1}(\Log(f)) = x_1 + x_1 x_2 + x_1 x_3 + x_1 x_2 x_3,$ and thus 
$V(f) =  \Log(x_2 + x_3 + x_2 x_3) = \{(0,1,0), (0,0,1), (0,1,1)\}$ by Theorem \ref{thm:invalg}.

\end{example}

Let $\phi$ be the map from the set of boolean polynomials in $n$ variables to itself defined by
$\phi(f)=x^{V(f)}$. 
It is a natural question to examine the consequences of the iterated use of $\phi$. As Theorem \ref{thm:fquad} will show, the fourth power of $\phi$ is the identity, which is a surprising result.

\begin{lemma} \label{lemma:fsquare}
Let $f$ be a boolean polynomial. Then
$$\phi^2(f) = f + \sum_{p \in \mathbb{Z}_2^n} x^p + 1.$$
\end{lemma}
\begin{proof}
It clearly holds that $x^{\Log(f)^c} = f + \sum_{p \in \mathbb{Z}_2^n} x^p$. Thus, if $1 \in \supp{f}$, then 
$x^{\Log(f)^c \cup \{(0,\ldots,0)\}} = f + \sum_{p \in \mathbb{Z}_2^n} x^p+1$, and if $1 \notin \supp{f}$, then
$x^{\Log(f)^c \setminus \{(0,\ldots,0)\}} = f + \sum_{p \in \mathbb{Z}_2^n} x^p+1$

The lemma now follows by taking the exponent on both sides in Lemma \ref{lemma:vanishing}.
\end{proof}

\begin{theorem} \label{thm:fquad}
Let $\phi$ be the map from the set of boolean polynomials in $n$ variables into itself defined by
$\phi(f)=x^{V(f)}$. Then
$\phi^4$ is the identity. 
\end{theorem}
\begin{proof}
Let $f$ be a boolean polynomial. Using Lemma \ref{lemma:fsquare}, we get
$$\phi^4(f) = \phi^2 (\phi^2 (f)) = \phi^2 ( \sum_{p \in \mathbb{Z}_2^n} x^p + f + 1) = 
\sum_{p \in \mathbb{Z}_2^n} x^p+ (\sum_{p \in \mathbb{Z}_2^n} x^p + f + 1) + 1 = f.
 $$
\end{proof}

We will end up by giving a result on the number of elements in the variety of a boolean polynomial.
\begin{proposition} \label{prop:odd}
Let $f$ be a boolean polynomial. Then $|V(f)|$ is odd if and only if 
$x_1 \cdots x_n \in \supp(f)$.
\end{proposition}

\begin{proof}

Let $V(f) = \{p_1,\ldots,p_s\}$. Then $f = 1 + S_{p_i} + \cdots + S_{p_s}$ 
since $V(f) = V(1 + S_{p_1} + \cdots + S_{p_s})$. Each 
$S_{p_i}$ contains a monomial of the form $x_1 \cdots x_n$. Hence $x_1 \cdots x_n \in \supp(f)$ if and only if $|V(f)|$ is odd. 

\end{proof}

\begin{remark}
In the context of error-correcting codes, Proposition \ref{prop:odd} gives an alternative proof of the fact that if $r<m$, then $R(r,m)$ consists of even codes only, see for instance \cite{introcoding}.
\end{remark}

\subsubsection{Efficiency of the methods compared to Gr\"obner methods} \label{sec:eff}
It is natural to ask if the methods that we have developed in this section are comparable in performance with other existing methods. This is somewhat out of the reach of this paper, so instead we focus on a special class of boolean polynomials. 

As we will see, a straight forward implementation of Corollary \ref{cor:main} would perform drastically better than the Buchberger algorithm when it comes to determining $|V(f)|$, where $f$ is an element in this class.

Indeed, let $u_1, \ldots, u_m$ be relatively prime quadratic boolean monomials in $\mathbb{Z}_2[x_1,\ldots, x_n]$, where $n \geq 2m$. Let $f = u_1 + \cdots + u_m$ and let $I$ be the boolean ideal generated by $f$. By Corollary 
\ref{cor:main}, $$|V(f)| = 2^n - \binom{m}{1} \cdot 2^{n-2} + 2 \binom{m}{2} \cdot 2^{n-4} + \cdots + 2^{m-1} (-1)^{m} \binom{m}{m} \cdot 2^{n-2m}.$$ This expression may be simplified to
$$2^{n-2m}(2^{m-1}(2^{m+1} - 2^m +1)) = 2^{n-2m}(2^{2m-1} +2^{m-1})$$ using the binomial theorem and the identity $1 = (2-1)^m$.

Recall that the leading monomials of the elements in a Gr\"obner basis of an ideal $I$ with respect to some monomial order $\prec$ generates the initial of 
$I$ with respect to $\prec$. The set of monomials outside this monomial ideal is called the 
\emph{standard monomials} with respect to $I$ and $\prec$. When $I$ is zero dimensional and radical, the number of standard monomials equals $|V(I)|$.

Thus, to determine $|V(f)|$, one computes a Gr\"obner basis of the ideal 
$I= (x_1^2-x_1,\ldots,x_n^2-x_n,f)$ with respect to some monomial order $\prec$, and then counts the number of standard monomials of $I$ with respect to $\prec$.

In Proposition \ref{prop:monster}, we show that the Gr\"obner basis of $(x_1^2-x_1,\ldots,x_n^2-x_n,u_1 + \cdots + u_m)$ contains $2^m-1$ elements independent of monomial ordering, and the total number of monomials in the reduced Gr\"obner basis is
$\sum_{i=1}^{m} 2^{2i-2} (m-i+1)$.
Although a naive implementation of the lcm-lattice in order to compute the number $|V(u_1 + \cdots + u_m)|$ using Corollary \ref{cor:main} would require $2^m$ lcm-operations, the Buchberger algorithm with a reduction process for each computed S-polynomial, would require drastically more memory and operations than the lcm technique. 

\begin{lemma} \label{lemma:gbasis}
Let $u_1 = x_1 x_2, u_2 = x_3 x_4, \ldots, u_m = x_{2m-1} x_{2m}$ be boolean monomials in $\mathbb{Z}_2[x_1, \ldots, x_{2m}]$, let $f = u_1 +\cdots + u_m$ and let $\prec$ be any monomial ordering such that $u_1 \succ \cdots \succ u_m$. Let
\begin{align*}
G_0(f) & = \{x_1^2-x_1,\ldots,x_n^2-x_n\},\\
G_1(f) &= \{u_1 + \cdots + u_m\}, \\
G_2(f) &= \{ (u_2+\cdots + u_{m}) \cdot (x_{i_1}+1), \, : \, x_{i_1} \in \supp(u_1) \}, \\
G_3(f)&=  \{ (u_3+\cdots + u_{m}) \cdot (x_{i_1}+1) \cdot (x_{i_2}+1) \, : \, x_{i_1} \in \supp(u_1), x_{i_2} \in \supp( u_2)\}, \\
\vdots \\
G_m(f) &= \{u_m \cdot  (x_{i_1} + 1) \cdots (x_{i_{m-1}} + 1)\, : \, x_{i_1} \in \supp(u_1), \ldots, x_{i_{m-1}} \in \supp(u_{m-1}) \}.  \\
\end{align*}
Then $\cup G_i(f)$ is a reduced, minimal Gr\"obner basis of $(x_1^2-x_1,\ldots,x_n^2-x_n,f)$ with respect to $\prec$ consisting of $2^m-1+n$ elements and a total of $\sum_{i=1}^{m} 2^{2i-2} (m-i+1)+2n$ monomials.
\end{lemma}

Before we give the proof, we give an example which hopefully makes the picture clear.

\begin{example}
Consider the boolean polynomial $f = x_1 x_2 + x_3 x_4 + x_5 x_6 \in \mathbb{Z}_2[x_1,\ldots,x_6]$ and the lexicographical order induced by $x_1 \succ \cdots  \succ x_6$. Then $G_0(f) = \{x_1^2-x_1,x_2^2-x_2,x_3^2-x_3,x_4^2-x_4,x_5^2-x_5,x_6^2-x_6\}$,
$G_1(f) = \{x_1 x_2 + x_3 x_4 + x_5 x_6\}, G_2(f) = \{(x_3 x_4 + x_5 x_6) \cdot (x_1+1), (x_3 x_4 + x_5 x_6) \cdot (x_2+1)\}$ and  $G_3(f) = \{x_5 x_6 \cdot (x_1 + 1) \cdot (x_3 + 1), x_5 x_6 \cdot (x_1 +1)\cdot (x_4+1), x_5 x_6 \cdot (x_2 + 1)\cdot (x_3+1), x_5 x_6 \cdot (x_2 + 1) \cdot (x_4+1)\}$. We have $|G| = 2^3 - 1 + 6 = 13$ and the number of monomials in $G$ is  $\sum_{i=1}^{3} 2^{2i-2} (3-i+1)+2 \cdot 6 = 33.$
\end{example}

\begin{proof}[Proof of Lemma \ref{lemma:gbasis}]
Let $G(f) = \cup G_i(f)$ and let $I$ be the boolean ideal generated by $f$. The idea is to show that $G(f) \subseteq I$ and that there are $|V(f)|$ standard monomials with respect to the monomial ideal generated by the leading monomials of the elements in $G$. From this it follows that $G$ is a 
Gr\"obner basis. 

Suppose that we have proven that $G(f)$ is a Gr\"obner basis. That $G(f)$ is reduced is clear since by construction, no leading elements in $G(f)$ divide any other monomials occurring in polynomials in $G(f)$. That $G(f)$ is minimal follows by a weaker form of the previous argument: no leading element in $G(f)$ is divisible by any other leading element in $G(f)$. The number of elements in each $G_i(f)$, $i>0$, is equal to $2^{i-1}$, so the number of elements in $G(f)$ is $2^m - 1 +n$. The number of boolean monomials in each element in $G_i(f)$, $i > 0$, is $2^{i-1} (m-i+1)$, so the number of monomials in $G(f)$ is $\sum_{i=1}^{m} 2^{2i-2} (m-i+1) +2n$.

It rests to show that $G(f)$ is indeed a Gr\"obner basis.  
Let $h= (u_s+\cdots +u_m)(x_{i_1}+1) \cdots  (x_{i_{s-1}} + 1)$ be an arbitrary element in $G_s(f)$, $s > 0$. Since 
$x_{i_k} | u_k$, we have $ (x_{i_k} + 1) \cdot u_k = 0$ as boolean polynomials, so
$$hf = h \cdot (u_1 + \cdots + u_{s-1}) + h \cdot (u_{s} + \cdots +u_m) =  0+h = h$$ as boolean polynomials. In follows that we can use Proposition \ref{prop:equivalances}  to conclude that $h \in I$, from which it follows that $G(f) \subseteq I$.

Finally, to prove that the number of standard monomials with respect to the monomial ideal generated by the leading monomials of the elements in $G(f)$ is equal to $|V(f)|$, we proceed by induction on $m$.

When $m = 1, G(x_1 x_2) = \{x_1^2-x_1,x_2^2-x_2,x_1 x_2 \}$ and the set of standard monomials is $\{1,x_1,x_2\}$. Since $|V(x_1 x_2)| = 2^2 - 2^0 =3$, it follows that $G(x_1 x_2)$ is a Gr\"obner basis of $(x_1^2-x_1,x_2^2-x_2,x_1 x_2)$.

Suppose that the lemma holds true when $m = t$, that is, $f = u_1 + \cdots + u_t$ and $G(f)$ is a Gr\"obner basis of 
$(x_1^2-x_1,\ldots,x_{2t}^2-x_{2t},f)$ in $\mathbb{Z}_2[x_1,\ldots, x_{2t}]$ We will now show that 
$G(f+u_{t+1})$ is a Gr\"obner basis of  $(x_1^2-x_1,\ldots, x_{2t+2}^2-x_{2t+2}, f+u_{t+1})$ 
in $\mathbb{Z}_2[x_1,\ldots,x_{2t+2}]$. Let
\begin{align*}
S_{0,0}& = \{s : s \text{ is a standard monomial in $G(f)$} \},\\
S_{1,0}& = \{s \cdot x_{2t+1} : s \text{ is a standard monomial in $G(f)$} \},\\
S_{0,1}& = \{s \cdot x_{2t+2} : s \text{ is a standard monomial in $G(f)$} \},\\
S_{1,1}& = \{ x_1^{1-\alpha_1} \cdots x_{2t}^{1-\alpha_{2t}} \cdot x_{2t+1} \cdot x_{2t+2} : x^{\alpha} \text{ is a leading boolean monomial in 
$G(f)$} \}.
\end{align*}

We claim that $S = S_{0,0} \cup S_{0,1} \cup S_{1,0} \cup S_{1,1}$ is the set of standard monomials with respect to 
$G(f + x_{2t+1} \cdot x_{2t+2})$. Clearly the four sets are disjoint, and the equalities 
$|S_{0,0}| = |S_{1,0}| = |S_{0,1}| = 2^{2t-1} + 2^{t-1}$ and $S_{1,1} = 2^{2t-1} - 2^{t-1}$ hold by the induction assumption. From this it follows that  $|S| = 
2^{2t+1} + 2^t$. Thus, we are done if we can show that each set contains standard monomials only.

The leading boolean monomials of $G_i(f)$ equals the leading elements of $G_i(f+u_{t+1})$ when $i \leq t$. The leading elements of $G_{t+1}(f+u_{t+1})$ is $u_{t+1} \cdots x_{i_1} \cdots x_{i_t}$, $x_{i_j} \in \supp(u_{i_j})$. It follows that a standard monomial of $G(f)$ is a standard monomial of $G(f+u_{t+1})$, and if $s$ is a standard monomial of $G(f)$, then both $s \cdot x_{2t+1}$ and $s \cdot x_{2t+2}$ are standard monomials of $G(f+ u_{t+1})$. Thus $S_{0,0}, S_{1,0}, S_{0,1}$ contains standard monomials.  

If $x^{\alpha}$ is a leading boolean monomial in $G(f)$, then $\alpha_{2i} = \alpha_{2i+1} = 1$ for some $i$, and
$\alpha_{2j} = 0 \implies \alpha_{2j + 1}=1$ when $j < i$. Thus, if $\beta = 1 - \alpha_1,\ldots, 1- \alpha_{2t}$, then 
$\beta_{2i} = \beta_{2i+1} = 0$ and $\beta_{2j} = 1 \implies \beta_{2j+1} = 0$ for $j<i$. Hence $\beta$ is not a leading monomial, so it is a standard monomial of $G(f)$. By the preceding argument, it is also a standard monomial of 
$G(f+ u_{t+1})$. But since $\beta_{2i} = \beta_{2i+1} = 0$ for some $i$, also the monomial $x^{\beta} \cdot x_{2t+1} \cdot x_{2t+2}$ is standard. This shows that also $S_{1,1}$ contains standard monomials. The lemma now follows by induction.

\end{proof}

\begin{proposition} \label{prop:monster}
Let $u_1, \ldots, u_m$ be relatively prime quadratic boolean monomials in $\mathbb{Z}_2[x_1,\ldots,x_n]$. Let $\prec$ be any monomial ordering. 
Let $G$ be the reduced minimal Gr\"obner basis of $(x_1^2-x_1,\ldots,x_n^2-x_n,f)$ with respect to $\prec$. Then $|G| = 2^m-1+n$ and $G$ consists of $\sum_{i=1}^{m} 2^{2i-2} (m-i+1) + 2n$ boolean monomials.

\end{proposition}
\begin{proof}
Since we assume the $u_i$:s to be relatively prime, we have $2m \leq n.$ Without loss of generality, we can assume that $u_1 = x_1 \cdot x_2, u_2= x_3 \cdot x_4, \ldots, u_m = x_{2m-1} x_{2m}$. 

The special case $2m = n$ now follows from Lemma \ref{lemma:gbasis}. 

When $n > 2m$, the Gr\"obner basis of $(x_1^2-x_1,\ldots,x_n^2-x_n,f)$ is equal to the Gr\"obner basis of 
$(x_1^2-x_1,\ldots,x_{2m}^2-x_{2m},f)$ in $\mathbb{Z}_2[x_1,\ldots,x_m]$, together with the remaining polynomials $x_{2m+1}^2-x_{2m+1}, \ldots, x_n^2-x_n$. This can be seen by considering the S-polynomials.  Pairs, where both polynomials come from the Gr\"obner basis of $(x_1^2-x_1,\ldots,x_{2m}^2-x_{2m},f)$ in $\mathbb{Z}_2[x_1,\ldots,x_{2m}]$ obviously reduce to zero, while pairs containing a polynomial of the form $x_i^2-x_i$, $i > 2m$, reduce to zero by the prime criteria, see \cite{cox}. It is clear that this constructed Gr\"obner basis is both reduced and minimal.
It follows from Lemma \ref{lemma:gbasis} that $|G| = 2^m-1+m + (n-m)$ and that $G$ consists of $\sum_{i=1}^{m} 2^{2i-2} (m-i+1) + 2m + 2(n-m)$ boolean monomials.

\end{proof}

\subsection{Standard monomials with respect to the lexicographical ordering} \label{sec:standard}

As was first shown in \cite{cm}, the standard monomials with respect to the lexicographical ordering (lex) induced by $x_1 \succ \cdots \succ x_n$ of a vanishing ideal of affine points is a combinatorial object in that the standard monomials depend only on the positions were the coordinates of the points differ.

It is natural to try to give a description of the lex standard monomials in terms of the defining polynomial of a boolean ideal. Such a description is given in Corollary \ref{cor:std1}. In Theorem \ref{thm:rec}, we give a computable recursive formula that given a boolean monomial determines the number of boolean ideals having this boolean monomial as a standard monomial, and in Corollary \ref{cor:statements}, we show that this number is odd.

Now some notation. Let $\Bbbk$ be a field and let $P$ be a set of distinct points in $\Bbbk^n$. Let $\{\alpha_1, \ldots, \alpha_h\}$ be the set of $n$'th coordinates of the points in $P$.  Let $P_i$ denote the set of points in $P$ whose $n$'th coordinate is equal to $\alpha_i$. Partition $P$ into the $h$ disjoint non-empty subsets 
$P_1 \cup \cdots \cup P_h$. Denote by $\overline{P}$ the projection of $P$ to $\Bbbk^{n-1}$ by omitting the last coordinate.

\begin{lemma} \label{lemma:stdunion}
Let $SM_i$ be the set of standard monomials of $I(\overline{P_i})$ with respect to lex and $x_1 \succ \cdots \succ x_{n-1}$. Then 
$$\{ x_n^j \cdot m \, | \, m \text{ occurs in $j+1$ number of the $SM_i$'s} \}$$ 
is the set of standard monomials for $I(P)$ with respect to lex and  $x_1 \succ \cdots \succ x_{n}$. 
\end{lemma}

\begin{proof}
The lemma is just a restatement of Corollary 8 (i) in \cite{lexgame}.
\end{proof}

We will now use Lemma \ref{lemma:stdunion} to derive a result which we have not been able to find in the literature. In this paper, we will use it to derive two corollaries. But it is formulated for any finite field, and might be of interest in its own right. Therefore we state it as a theorem.

\begin{theorem} \label{thm:finitefieldstd}
Let $I$ be a radical zero-dimensional ideal in $\Bbbk[x_1,\ldots,x_n]$, where $\Bbbk = \{\alpha_1,\ldots, \alpha_q\}$ is a finite field and suppose that $V(I)$ is rational. Let 
$SM_i^{n}$ be the set of standard monomials of $(I,x_n - \alpha_i)$ with respect to the lexicographical ordering induced by $x_1 \succ \cdots \succ x_{n}$. Then 
$$SM(I) = \{ x_n^j \cdot m \, | \, m \text{ occurs in $j+1$ number of the $SM_i^n$'s} \}.$$ 
\end{theorem}

\begin{proof}
According to the assumptions, $I = I(P)$, where $P$ is a set of points in $\Bbbk^n$. The equality $(I,x_n - \alpha_i) = (1)$ holds if and only if there are no points $p \in P$ such that $p_n = \alpha_i$. Moreover $(I,x_n - \alpha_i) = (1)$ if and only if $SM_i^n = \{ \}$. 

Suppose, without loss of generality, that $(I,x_n - \alpha_i) \neq (1)$ for $i = 1, \ldots, h$, and
$(I,x_n - \alpha_i) = (1)$ for $i = h+1, \ldots, q$. Thus it is enough to show that 

$$SM(I) = \{ x_n^j \cdot m \, | \, m \text{ occurs in $j+1$ number of the $SM_i^n$'s, $i \leq h$} \}.$$

If $P_i$ is the set of points in $P$ whose $n$'th coordinate is equal to $\alpha_i$, then $(I,x_n - \alpha_i) = I(P_i)$. If we can prove that the standard monomials $SM_i$ of $I(\overline{P_i}) \subseteq \Bbbk[x_1,\ldots,x_{n-1}]$ agree with the standard monomials $SM_i^n$ of $I(P_i) \subseteq \Bbbk[x_1,\ldots,x_n]$, then we are done by Lemma \ref{lemma:stdunion}. But if $a \in I(\overline{P_i})$, then $a \in I(P_i)$, so a leading monomial in $I(\overline{P_i})$ is a leading monomial in $ I(P_i)$. Moreover, $x_n$ is a leading monomial in $I(P_i)$. Thus, $SM_i^n \subseteq SM_i$. Since $|P_i| = |\overline{P_i}|$, it must in fact hold that $SM_i^n = SM_i$.

\end{proof}

Let $S = \{s_1,\ldots, s_m\}$ be a subset of $\mathbb{Z}_2[x_1,\ldots,x_n]$. Then $f \cdot S$ is defined to be 
$\{f \cdot s_1,\ldots, f \cdot s_m\}$.

\begin{corollary} \label{cor:std1}
Let $I$ be a boolean ideal with defining polynomial $f$. Then 
\begin{align*}
&SM((f, x_1^2-x_1, \ldots, x_n^2-x_n)) =\\
&SM((f(x_1,\ldots, x_{n-1},0), x_1^2-x_1,\ldots, x_{n-1}^2-x_{n-1})  \cup \\
&SM((f(x_1,\ldots, x_{n-1},1), x_1^2-x_1,\ldots, x_{n-1}^2-x_{n-1}) \cup \\
&x_n \cdot (SM((f(x_1,\ldots, x_{n-1},0), x_1^2-x_1,\ldots, x_{n-1}^2-x_{n-1})) \cap \\
& SM((f(x_1,\ldots, x_{n-1},1), x_1^2-x_1,\ldots, x_{n-1}^2-x_{n-1})),
\end{align*}
where the standard monomials on the right hand side are understood to be elements in $\mathbb{Z}_2[x_1,\ldots, x_{n-1}].$
\end{corollary}

\begin{proof}
The ideal $I$ is radical and $V(I)$ is rational. Hence 
$$SM(I) = \{ m : m \text{ occurs in } SM_0 \text{ or in } SM_1 \} \cup \{ x_n m :  m \text{ occurs in } SM_0 \text{ and in } SM_1\}$$
by Theorem \ref{thm:finitefieldstd}.
\end{proof}

\begin{corollary} \label{cor:std2}
The monomial $x^{\alpha}$, where $\alpha_n = 1$, is standard with respect to $I(P)$ if and only if 
$x_1^{\alpha_1} \cdots x_{n-1}^{\alpha_ {n-1}}$ is standard with respect to $I(\overline{P_0})$ and $x_1^{\alpha_1} \cdots x_{n-1}^{\alpha_ {n-1}}$ is standard with respect to $I(\overline{P_1})$.
The monomial $ x^{\alpha}$, where $\alpha_n = 0$, is standard with respect to $P$ if and only if 
$x_1^{\alpha_1} \cdots x_{n-1}^{\alpha_ {n-1}}$ is standard with respect to $I(\overline{P_0})$ or $x_1^{\alpha_1} \cdots x_{n-1}^{\alpha_ {n-1}}$  is standard with respect to $I(\overline{P_1})$.
\end{corollary}
\begin{proof}
This is a transposed version of Corollary \ref{cor:std1}.
\end{proof}

An application of Corollary \ref{cor:std2} is the association of a binary decision tree to each standard monomial $x^{\alpha}$, exemplified in Figure \ref{fig:tree}. The $i$'th level of the tree consists of OR gates if $\alpha_{n-i} = 0$, and of AND gates if $\alpha_{n-i} = 1.$ 

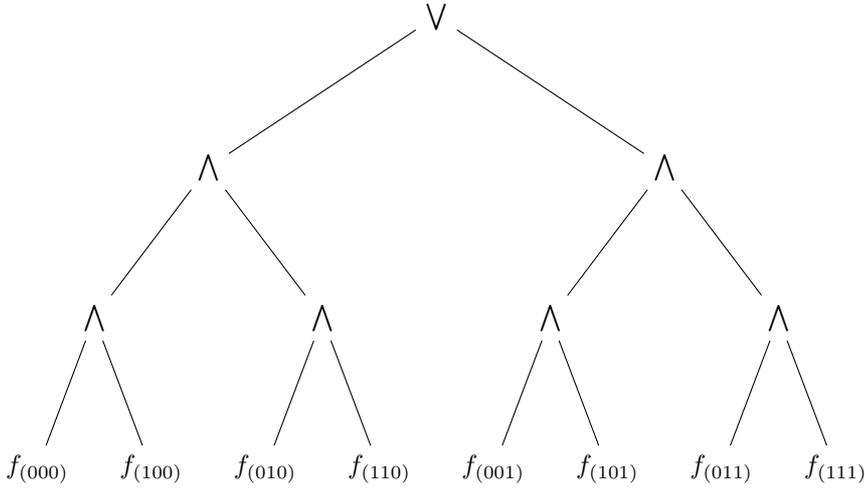
\begin{figure}[ht]  
\begin{tikzpicture} 
[level 3/.style={sibling distance=15mm},level distance=20mm,level 2/.style={sibling distance=30mm}, level 1/.style={sibling distance=60mm}]
\node {$\bigvee$}
	child {node {$\bigwedge$}
		child {node {$\bigwedge$}
			child {node {$f_{(000)}$}}
			child {node {$f_{(100)}$}}}
		child {node {$\bigwedge$}
			child {node {$f_{(010)}$}}
			child {node {$f_{(110)}$}}}}
	child {node {$\bigwedge$}
		child {node {$\bigwedge$}
			child {node {$f_{(001)}$}}
			child {node {$f_{(101)}$}}}
		child {node {$\bigwedge$}
			child {node {$f_{(011)}$}}
			child {node {$f_{(111)}$}}}};
			
\end{tikzpicture}
\caption{The tree for checking if the monomial $x_1 x_2$ is standard with respect to lex and $x_1 \succ x_2 \succ x_3$, where $f_{(0,0,0)}$ is defined to be $ f(0,0,0) + 1$ and where $a \wedge b = 1$ if (and only if) 
$a = b = 1$ and $a \vee b = 0$ if (and only if) $a= b=0$. Accordingly, the monomial $x_1 x_2$ is standard when the tree evaluates to $1$. The number of solutions in $\mathbb{Z}_2[f_{(000)}, f_{(001)}, \ldots, f_{(111)}]$ to this equation is $31$, 
so there are $31$ out of $256$ varieties/boolean ideals that contain $x_1 x_2$ as a standard monomial. Of course, this number can be directly computed using the function $\Omega$ from Theorem \ref{thm:rec}, i.e. $\Omega((1,1,0)) = 31.$ } \label{fig:tree}
\end{figure}

\begin{lemma} \label{lemma:booleanbijection}
Let $P(n)$ be the set of boolean polynomials in $\mathbb{Z}_2[x_1,\ldots,x_n]$.
Then the map $P(n-1) \times P(n-1) \to P(n)$, $(g_1,g_2) \mapsto g_1 \cdot x_n + g_2 \cdot (1+x_n)$ is a bijection. 
\end{lemma}
\begin{proof}
We have $|P(n)| = 2^{2^n} = 2^{2^{n-1}} \cdot 2^{2^{n-1}} = |P(n-1) \times P(n-1)|$, so it is enough to show that the map is injective.
Suppose that $g_1 \cdot x_n + g_2 \cdot (1+x_n) = g_3 \cdot x_n + g_4 \cdot (1+x_n)$. Then
$g_2 = g_4$, and this implies that $g_1 = g_3$. This shows that the map is injective, so it is a bijection.
\end{proof}

We are now ready to state the main theorem of this section.

\begin{theorem} \label{thm:rec}
Let $n>0$, let $\alpha$ be an element in $\{0,1\}^n$, let $\Omega(()) = 1$ and let $\Omega((\alpha_1,\ldots, \alpha_n))$ denote the number of boolean ideals such that $x^{\alpha}$ is a standard monomial with respect to lex and $x_1 \succ \cdots \succ x_n$. Then 
\begin{equation*}
\Omega((\alpha_1,\ldots,\alpha_n)) = 
\begin{cases}
\Omega(\overline{\alpha})^2 & \text{if }  \alpha_n = 1, \\
 2 \cdot \Omega(\overline{\alpha}) \cdot 2^{2^{n-1}} -\Omega(\overline{\alpha})^2 & \text{if } \alpha_n = 0. \\
\end{cases}
\end{equation*}
where $\overline{(\alpha_1,\ldots, \alpha_n)} = (\alpha_1,\ldots,\alpha_{n-1}).$

\end{theorem}

\begin{proof}
If $n = 1$, then $x_1$ is a standard monomial of the ideal $(x_1^2-x_1,0)$, while $1$ is a standard monomial of the ideals
$(x_1^2-x_1,0), (x_1^2-x_1,x_1)$ and  $(x_1^2-x_1,x_1+1).$ Thus $\Omega((1)) = \Omega(())^2$ and $\Omega((0)) = 2 \cdot \Omega(()) \cdot 2^{2^0} - \Omega(())^2.$

Suppose that $n > 1$.
By Lemma \ref{lemma:booleanbijection}, each boolean polynomial $f \in \mathbb{Z}_2[x_1,\ldots,x_n]$ can be uniquely written on the form 
$g_1 \cdot x_n + g_2 \cdot (1+x_n)$, with $g_1,g_2 \in \mathbb{Z}_2[x_1,\ldots,x_{n-1}]$.

Suppose that $\alpha_n = 1$. By Corollary \ref{cor:std2}, the monomial $x_1^{\alpha_1} \cdots x_{n}^{\alpha_{n}}$ is standard with respect to the boolean ideal generated by $f$
 if and only if $x_1^{\alpha_1} \cdots x_{n-1}^{\alpha_{n-1}}$ is standard with respect to the boolean ideal generated by $f(x_1,\ldots, x_{n-1},0) = g_2$ and
  $x_1^{\alpha_1} \cdots x_{n-1}^{\alpha_{n-1}}$ is standard with respect to the boolean ideal generated by $f(x_1,\ldots, x_{n-1},1)=g_1$.
  Thus the set of polynomials such that $x^{\alpha}$ is standard with respect to the associated boolean ideal is
 $$ \{g_1 \cdot x_n + g_2 \cdot (1+x_n) : x_1^{\alpha_1} \cdots x_{n-1}^{\alpha_{n-1}} \text{ is standard with respect to } g_1 \text{ and }g_2 \}.$$ 
 By Lemma \ref{lemma:booleanbijection}, the cardinality of this set is
$ \Omega(\overline{\alpha})^2$.
   
  Suppose that $\alpha_n = 0$. By Corollary \ref{cor:std2}, the monomial $x_1^{\alpha_1} \cdots x_{n}^{\alpha_{n}}$ is standard with respect to the boolean ideal generated by $f$
 if and only if $x_1^{\alpha_1} \cdots x_{n-1}^{\alpha_{n-1}}$ is standard with respect to the boolean ideal generated by $f(x_1,\ldots, x_{n-1},0) = g_2$ or
  $x_1^{\alpha_1} \cdots x_{n-1}^{\alpha_{n-1}}$ is standard with respect to the boolean ideal generated by $f(x_1,\ldots, x_{n-1},1)=g_1$. 
  Thus, the set of boolean polynomials is equal to 
  $$\{g_1\cdot x_n + g_2 \cdot ( 1+x_n)  : x_1^{\alpha_1} \cdots x_{n-1}^{\alpha_{n-1}} \text{ is standard with respect to }
  g_1 \text{ or } g_2\}.$$ 
  It follows that there are 
  $2^{2^{n-1}} \cdot 2^{2^{n-1}} - ( 2^{2^{n-1}} - \Omega(\overline{\alpha})) \cdot ( 2^{2^{n-1}} - \Omega(\overline{\alpha})) = 2 \cdot 2^{2^{n-1}} \Omega(\overline{\alpha}) -\Omega(\overline{\alpha})^2$ such boolean polynomials, where we
 have used that the number of boolean polynomials in $\mathbb{Z}_2[x_1,\ldots,x_{n-1}]$ such that $x_1^{\alpha_1} \cdots x_{n-1}^{\alpha_{n-1}}$ is not a standard monomial of the associated boolean ideal, is equal to 
$2^{2^{n-1}} - \Omega(\overline{\alpha})$.

\end{proof}

\begin{remark}
The recursive formula in Theorem \ref{thm:rec} can be rewritten as
\begin{equation*}
\Omega((\alpha_1,\ldots,\alpha_n)) = 
\begin{cases}
1 & \text{if } n = 0,\\  
(1-\alpha_n) \cdot \Omega(\overline{\alpha}) \cdot 2 \cdot 2^{2^{n-1}} - (-1)^{\alpha_n} \cdot  
 \Omega(\overline{\alpha})^2 & \text{otherwise.}\\
 \end{cases}
 \end{equation*}
\end{remark}

Although we have not been able to obtain a closed form of $\Omega$, the recursive formula can be used to prove statements about $\Omega$ based on induction.

\begin{corollary} \label{cor:statements}
\hspace{1cm}
\begin{enumerate}

\item $\Omega(\alpha)$ is odd.
\item $\Omega((\alpha_1,\ldots, \alpha_n)) = 2^{2^n} - \Omega((1-\alpha_1,\ldots, 1-\alpha_n)).$
\item $\sum_{\alpha \in \mathbb{Z}_2^n} \Omega(\alpha) = 2^{2^n + n - 1}$.
\end{enumerate}
\end{corollary}

\begin{proof} The proof of the first and the second statement is done by induction over $n$. The first is trivial. Here is the induction step for the second statement.
Let $\beta = (1-\alpha_1,\ldots, 1- \alpha_n)$. 
Suppose that $\alpha_n = 0$. Then 
$$\Omega(\beta) = \Omega(\overline{\beta})^2 = 2^{2^n} + \Omega(\overline{\alpha})^2 - 2 \cdot \Omega(\overline(\alpha)) \cdot 2^{2^{n-1}}= 2^{2^n} - \Omega(\alpha).$$
Suppose instead that $\alpha_n = 1$. Then
\begin{align*}
\Omega(\beta)
&= -\Omega(\overline{\beta})^2 + 2 \cdot \Omega(\overline{\beta}) \cdot 2^{2^{n-1}} \\
&=- ((2^{2^{n-1}} -  \Omega(\overline{\alpha}))^2 + 2 \cdot (2^{2^{n-1}} - \Omega(\overline{\alpha})) \cdot 2^{2^{n-1}}\\
&=-2^{2^n} - \Omega(\overline{\alpha})^2 + 2 \cdot \Omega(\overline{\alpha}) \cdot 2^{2^{n-1}} + 2 \cdot 2^{2^{n}} -2 \cdot \Omega(\overline{\alpha}) \cdot 2^{2^{n-1}}  \\ 
&= 2^{2^n} - \Omega(\overline{\alpha})^2  \\
&=2^{2^n} - \Omega(\alpha).
\end{align*}

The third statement follows from the second by grouping complementary exponent vectors together.

\end{proof}

\begin{remark}
The sequence $\sum_{\alpha \in \mathbb{Z}_2^i} \Omega(\alpha)$ for $i = 0,\ldots$ is equal to the sequence A028369 in OEIS.
\end{remark}

\section{Conclusion and further work}
We have obtained several ways to express the variety of a boolean polynomial (Theorem \ref{thm:explicit1}, Theorem \ref{thm:explicit2}, Corollary \ref{cor:main}, Theorem \ref{thm:invalg}), and in Theorem \ref{thm:fquad} we were able to prove the surprising fundamental identity $\phi^4(f) = f$.

Extensions of these results to other finite fields seem hard, but experiments show that Theorem \ref{thm:fquad} might be able to generalize for certain classes of polynomials over the field $\mathbb{Z}_3$.

When it comes to the standard monomials with respect to lex, we have been able to give a description of the distribution of the standard monomials among the set of boolean ideals. A natural continuation of this work is to examine whether it is possible to use Theorem \ref{thm:finitefieldstd} in order to describe the distribution of the standard monomials with respect to vanishing ideals of points over other finite fields than $\mathbb{Z}_2$. Another interesting problem is to determine under which circumstances it is possible to obtain the set of standard monomials of $\phi(f)$ from the set of standard monomials of $f$.

From a computational point of view, and as indicated by the performance analysis in Section \ref{sec:eff}, a natural next step is to implement and compare these methods with each other and with state of the art software for boolean ideals based upon the Buchberger algorithm, for instance PolyBoRi \cite{polybori}. One would expect that when the defining polynomial has a low number of terms, the methods proposed in this paper perform better, while if the defining polynomial has a large number of terms, the Buchberger algorithm would perform better. 
Also, the methods for computing the variety presented in this paper have interesting connections to the resolutions of square free monomial ideals, so techniques developed to minimize a given monomial resolution might be useful when implementing the proposed algorithms.

Keeping the computational perspective, a major weakness of the methods proposed in this paper
is that if we are given a system of polynomials as input, then we first have to compute the defining boolean polynomial, a computation that leads to a major term expansion. On the other hand, the problem of determining whether $(f_1 + 1 )  \cdots (f_m+1) + 1 = 1$ as boolean polynomials does not necessarily need to be equivalent to such a term expansion on the left hand side, and might be worth to look into, perhaps in connection to the polynomial identity problem, see \cite{pit}.


\begin{thebibliography}{99}

\bibitem{atiyah}
M.F. Atiyah and I.G. MacDonald, Introduction to Commutative Algebra, Addison-Wesley. Reading MA, 1969.

\bibitem{bardet}
M. Bardet, J.-C. Faug\`{e}re, B. Salvy, P.-J. Spaenlehauer, On the complexity of solving quadratic Boolean systems. J. Complexity 29 (2013) no. 1, 53 -- 75.

\bibitem{polybori}  M. Brickenstein, A. Dreyer, PolyBoRi: A framework for Gr\"obner-basis computations with Boolean polynomials. J. Symb. Comput. 44 (2009), no. 9, 1326 -- 1345.

\bibitem{newdev}
M. Brickenstein, A. Dreyer, G.-M. Greuel, F. Seelisch, O. Wienand, New developments in the theory of Gr\"obner bases and applications to formal verification. J. Pure Appl. Alg. 213 (2009),  no. 8, 1612 -- 1635.


\bibitem{cm}
L. Cerlienco, M. Mureddu, From algebraic sets to monomial linear bases by means of combinatorial algorithms. Discrete Math. 139 (1995), 73 -- 87.

\bibitem{cox}
D. Cox, J. Little, D. O'Shea, Ideals, varieties, and algorithms, Springer, 2007.

\bibitem{lexgame}  
B. Felszeghy, B. R\'{a}th, L. R\'{o}nyai, The lex game and some applications. J. Symb. Comput. 41 (2006), no. 6, 663 -- 681.

\bibitem{froberg}
D. Ferrarello, R. Fr\"oberg, The Hilbert Series of the Clique Complex. Graphs and Combinatorics 21 (2005) no. 4, 401 -- 405.


\bibitem{gerdt}
V. Gerdt, M. Zinin, A pommaret division algorithm for computing Gr\"obner bases in boolean rings, ISSAC '08 (2008), 95 -- 102.


\bibitem{introcoding}
J. Hendricus V. Lint, Introduction to Coding Theory. Springer-Verlag New York, 1982. ISBN: 0387112847.

        
\bibitem{sato}
Y. Sato, S, Inoue,  A. Suzuki, K. Nabeshima, K. Sakai,  Boolean Gr\"obner bases. J. Symb. Comput. 46 (2011), no. 5,  622 -- 632.

\bibitem{pit}
J.T. Schwartz, Fast Probabilistic Algorithms for Verification of Polynomial Identities. Journal of the ACM 27 (2011), no. 4, 701 -- 717. 

\end{thebibliography}
\end{document}